\documentclass[12pt]{amsart}
\usepackage[cmtip,all]{xy}
\usepackage{graphicx}
\usepackage{amssymb}
\usepackage{mathrsfs}
\usepackage{amsfonts}
\usepackage{amsmath}

\newtheorem{theorem}{Theorem}[section]
\newtheorem{lemma}[theorem]{Lemma}

\theoremstyle{definition}

 \theoremstyle{remark}

\newtheorem{corollary}[theorem]{Corollary}

 \numberwithin{equation}{section}

\begin{document}

\title[Sharp Hardy-Adams  inequalities for bi-Laplacian on hyperbolic space]{ Sharp Hardy-Adams  inequalities for bi-Laplacian on hyperbolic space of dimension four}

\author{Guozhen Lu}
\address{Department of Mathematics,   University of Connecticut, CT 06269, USA}

\email{guozhen.lu@uconn.edu}

\author{Qiaohua  Yang}
\address{School of Mathematics and Statistics, Wuhan University, Wuhan, 430072, People's Republic of China}

\email{qhyang.math@gmail.com}

\thanks{The first author's research  was supported by  a US NSF grant  and a Simons Fellowship from the Simons Foundation. The second author's research was supported by   the National Natural
Science Foundation of China (No.11201346).  }






\begin{abstract}
We establish sharp Hardy-Adams  inequalities on hyperbolic space
$\mathbb{B}^{4}$ of dimension four. Namely,  we will show that for any $\alpha>0$ there exists a constant $C_{\alpha}>0$
such that
\[
\int_{\mathbb{B}^{4}}(e^{32\pi^{2} u^{2}}-1-32\pi^{2}
u^{2})dV=16\int_{\mathbb{B}^{4}}\frac{e^{32\pi^{2}
u^{2}}-1-32\pi^{2} u^{2}}{(1-|x|^{2})^{4}}dx\leq C_{\alpha}.
\]
for any $u\in C^{\infty}_{0}(\mathbb{B}^{4})$ with
\[
\int_{\mathbb{B}^{4}}\left(-\Delta_{\mathbb{H}}-\frac{9}{4}\right)(-\Delta_{\mathbb{H}}+\alpha)u\cdot
udV\leq1.
\]
 As   applications, we obtain a sharpened  Adams
inequality on hyperbolic space $\mathbb{B}^{4}$  and  an inequality  which improves the classical Adams' inequality and
the  Hardy inequality simultaneously. The later inequality is in the spirit of the
Hardy-Trudinger-Moser inequality on a disk in dimension two given by Wang and Ye \cite{wy} and on any convex planar domain by the authors
\cite{ly}.

The tools of fractional Laplacian, Fourier transform and the Plancherel formula  on hyperbolic and symmetric spaces play an important role  in our work.

\end{abstract}

\maketitle


\section{Introduction}

Our main purpose of this article is to establish   sharp Hardy-Adams  inequalities on hyperbolic space in dimension four
$\mathbb{B}^{4}$.

We first recall the classical Trudinger-Moser inequality in any finite domain of Euclidean spaces.
Let $\Omega\subset \mathbb{R}^n (n\ge2)$ be a bounded domain and  $1\le q\le\frac{np}{n-kp}$. Then   it is well known that the Sobolev embedding theorem tells us  the embedding $W^{k,p}_0(\Omega)\subset L^q(\Omega)$ is continuous when $kp<n$. However,   in general $W^{1, n}_0(\Omega)\nsubseteq L^{\infty}(\Omega)$.   Trudinger \cite{t} established  in the borderline case  that $W^{1, n}_0(\Omega)\subset L_{\varphi_n}(\Omega)$,  where $L_{\varphi_n}(\Omega)$ is the Orlicz space associated with the Young function $\varphi_n(t)=\exp(\beta|t|^{n/n-1})-1$ for some $\beta>0$ (see also Yudovich \cite{Yu}, Pohozaev \cite{Po}).  In 1971, Moser  sharpened  the Trudinger inequality  in   \cite{mo} by finding the optimal $\beta$:\\
\\
\begin{theorem}\label{theorem1.1}[Trudinger-Moser] Let $\Omega$ be a domain with finite measure in Euclidean n-space $\mathbb{R}^n$, $n\ge2$. Then there exists a sharp constant $\beta_{n}=n\left(
\frac{n\pi^{\frac{n}{2}}}{\Gamma(\frac{n}{2}+1)}\right)  ^{\frac{1}{n-1}}%
$ such that
\begin{displaymath}
\frac{1}{|\Omega|}\int_{\Omega}\exp(\beta|u|^{\frac{n}{n-1}})dx\le c_0
\end{displaymath}
for any $\beta\le\beta_n$, any $u\in W^{1,n}_0(\Omega)$ with $\int_{\Omega}|\nabla u|^ndx\le1$. This constant $\beta_n$ is sharp in the sense that if $\beta>\beta_n$, then the above inequality can no longer hold with some $c_0$ independent of $u$.
\end{theorem}

In 1988, D. Adams extended such an  inequality to
high order Sobolev spaces. In fact, Adams proved the following
theorem:
\begin{theorem}\label{theorem1.2}
Let $\Omega$ be a domain in $\mathbb{R}^{n}$ with finite $n$-measure
and $m$ be a positive integer less than $n$. There is a constant
$c_{0}=c_{0}(m,n)$ such that for all $u\in C^{m}(\mathbb{R}^{n})$
with support contained in $\Omega$ and $\|\nabla^{m}u\|_{n/m}\leq
1$,
 the following uniform
inequality holds
\begin{equation}\label{1.2}
\frac{1}{|\Omega|}\int_{\Omega}\exp(\beta_{0}(m,n)|u|^{n/(n-m)})dx\leq
c_{0},
\end{equation}
where
\begin{equation}\label{1.3}
\beta_{0}(m,n)=\left\{
                 \begin{array}{ll}
                   \frac{n}{\omega_{n-1}}\left[\frac{\pi^{n/2}2^{m}\Gamma((m+1)/2)}{\Gamma((n-m+1)/2)}\right]^{n/(n-m)}, & \hbox{$m$ = odd;} \\
                   \\
                    \frac{n}{\omega_{n-1}}\left[\frac{\pi^{n/2}2^{m}\Gamma(m/2)}{\Gamma((n-m)/2)}\right]^{n/(n-m)}, & \hbox{$m$ = even,}
                 \end{array}
               \right.
\end{equation}
where
$\omega_{n-1}$ is the surface measure of the unite sphere in
$\mathbb{R}^{n}$.

 Furthermore, the constant $\beta_{0}(m,n)$ in (\ref{1.2}) is sharp in
the sense that if $\beta_{0}(m,n)$ is replaced by any larger number,
then the integral in (\ref{1.2}) cannot be bounded uniformly by any
 constant.
\end{theorem}

Note that $\beta(1, n)$ coincides with Moser's value of $\beta_{n}$.  We are particularly interested in the case $n = 4$ and $m = 2$ in this paper
where $\beta_0(2, 4) = 32\pi^{2}$.

 There have been many generalizations related to the Trudinger-Moser inequality on hyperbolic spaces and Riemannian manifolds (see e.g.,  \cite{f},   \cite{ks},  \cite{l1}, \cite{l2}, \cite{yxli1}, \cite{yxli2}, \cite{ly}, \cite{ms}, \cite{mst}, \cite{y}, \cite{ysk}). For instance,  Mancini and Sandeep \cite{ms} 
proved the following improved Trudinger-Moser inequalities on $\mathbb{B}=\{z=x+iy:|z|=\sqrt{x^{2}+y^{2}}<1\}$:
\[
\sup_{u\in W^{1,2}_{0}(\mathbb{B}),
\int_{\mathbb{B}}|\nabla
u|^{2}dxdy\leq1}\int_{\mathbb{B}}\frac{e^{4\pi
u^{2}}-1}{(1-|z|^{2})^{2}}dxdy<\infty.
\]
 Later,  Karmakar and  Sandeep \cite{ks} generalize this inequality to hyperbolic space $\mathbb{H}^{n}$ if $n$ is even.
 In \cite{l1, l2}, the first author and Tang established  sharp critical and subcritical  Trudinger-Moser  inequalities on the high dimensional hyperbolic spaces which are different from those in \cite{mst}.  The results have been generalized by  Ng\^o and Nguyen \cite{n}, among other results, for bi-Laplacian on hyperbolic spaces.

 Wang and Ye \cite{wy} proved, among other results, an improved  Trudinger-Moser inequality by
combining the Hardy inequality. Their result is the following
\begin{theorem}\label{theorem1.3}
There exists a constant $C>0$ such that
\[
\int_{\mathbb{B}}e^{\frac{4\pi u^{2}}{\|u\|_{\mathcal{H}}}}dxdy<C<\infty,\;\;\forall u\in C^{\infty}_{0}(\mathbb{B}),
\]
where $\|u\|_{\mathcal{H}}=\int_{\mathbb{B}}|\nabla u|^{2}dxdy-\int_{\mathbb{B}}\frac{u^{2}}{(1-|z|^{2})^{2}}dxdy$.
\end{theorem}
We note that the  proof of Theorem \ref{theorem1.3} in \cite{wy} depends
on Schwartz rearrangement argument. In the same paper, they conjecture that such Hardy-Trudinger-Moser  inequality holds for bounded and convex domains with smooth boundary.
Using Theorem \ref{theorem1.3},  Mancini,  Sandeep and Tintarev \cite{mst} proved, among other results, the following modified Trudinger-Moser  inequality on $\mathbb{B}$
and their proof also depends on rearrangement inequalities.
\begin{theorem}\label{theorem1.4}
There exists a constant $C$ such that for all
$u\in C^{\infty}_{0}(\mathbb{B})$ with
\[
\|u\|_{\mathcal{H}}=\int_{\mathbb{B}}|\nabla u|^{2}dxdy-\int_{\mathbb{B}}\frac{u^{2}}{(1-|z|^{2})^{2}}dxdy\leq1,
\]
there holds
\[
\int_{\mathbb{B}}\frac{(e^{4\pi u^{2}}-1-4\pi u^{2})}{(1-|x|^{2})^{2}}dx\leq C.
\]
\end{theorem}

Recently, both authors confirm in \cite{ly} that the conjecture given by Wang and
Ye \cite{wy} indeed holds for any bounded and convex domain in
$\mathbb{R}^{2}$  via  the Riemann mapping theorem. More precisely, the authors established
in \cite{ly} the following:

\begin{theorem}
Let
$\Omega$ be a  bounded and convex domain in $\mathbb{R}^{2}$.
There exists a finite constant $C({\Omega})>0$ such that
\[
\int_{\Omega}e^{\frac{4\pi u^{2}}{H_{d}(u)}}dxdy\le C(\Omega),\;\;\forall u\in C^{\infty}_{0}(\Omega),
\]
where $H_{d}=\int_{\Omega}|\nabla u|^{2}dxdy-\frac{1}{4}\int_{\Omega}\frac{u^{2}}{d(z,\partial\Omega)^{2}}dxdy$ and $d(z,\partial\Omega)=\min\limits_{z_{1}\in\partial\Omega}|z-z_{1}|$.
\end{theorem}

It  then becomes a very interesting and highly nontrivial  question whether Theorem \ref{theorem1.3} holds for higher  order
derivatives. In this paper we shall show this is indeed the case on $n-$dimensional hyperbolic spaces $\mathbb{B}^n$
when $n=4$.

\medskip

To state our results, let us agree to some conventions.
We use the Poincar\'e model of the hyperbolic space $\mathbb{B}^{n}$. Recall
that the Poincar\'e  model is the unit ball
\[\mathbb{B}^{n}=\{x=(x_{1},\cdots,x_{n})\in \mathbb{R}^{n}| |x|<1\}\]
equipped with the usual Poincar\'e metric
\[
ds^{2}=\frac{4(dx^{2}_{1}+\cdots+dx^{2}_{n})}{(1-|x|^{2})^{2}}.
\]
The hyperbolic volume element is
\[
dV=\left(\frac{2}{1-|x|^{2}}\right)^{n}dx.\] The associated
Laplace-Beltrami operator is given by
\[
\Delta_{\mathbb{H}}=\frac{1-|x|^{2}}{4}\left\{(1-|x|^{2})\sum^{n}_{i=1}\frac{\partial^{2}}{\partial
x^{2}_{i}}+2(n-2)\sum^{n}_{i=1}x_{i}\frac{\partial}{\partial
x_{i}}\right\}.
\]
The GJMS operators on $\mathbb{B}^{n}$ are given by (see \cite{GJMS}, \cite{j})
\[
P_{1}(P_{1}+2)\cdots(P_{1}+k(k-1)), \;\; k\in \mathbb{N}\backslash\{0\},
\]
where $P_{1}=-\Delta_{\mathbb{H}}-\frac{n(n-2)}{4}$ is the conformal Laplacian on $\mathbb{B}^{n}$. In the case $n=4$ and $k=2$, the GJMS operator is nothing but the Paneitz operator on
$\mathbb{B}^{4}$ which satisfies (see  \cite{liu})
\[
P_{1}(P_{1}+2)=\left(\frac{1-|x|^{2}}{2}\right)^{4}\Delta^{2},
\]
where $\Delta=\sum\limits^{4}_{i=1}\frac{\partial^{2}}{\partial
x^{2}_{i}}$ is the Laplacian on $\mathbb{R}^{4}$. Therefore, for $u\in C^{\infty}_{0}(\mathbb{B}^{4})$,
\begin{equation}\label{b1.1}
\int_{\mathbb{B}^{4}}(-\Delta_{\mathbb{H}})(-\Delta_{\mathbb{H}}+2)u\cdot
udV=\int_{\mathbb{B}^{4}}|\Delta
u|^{2}dx.
\end{equation}

It is known that the  spectral gap of $-\Delta_{\mathbb{H}}$ on
$L^{2}(\mathbb{B}^{n})$ is $\frac{(n-1)^{2}}{4}$ (see e.g.
\cite{m1}), i.e.
\begin{equation}\label{1.4}
\int_{\mathbb{B}^{n}}|\nabla_{\mathbb{H}}u|^{2}dV\geq\frac{(n-1)^{2}}{4}\int_{\mathbb{B}^{n}}u^{2}dV,\;\;\;
u\in C^{\infty}_{0}(\mathbb{B}^{n}),
\end{equation}
and the constant $\frac{(n-1)^{2}}{4}$ is sharp. Therefore, by (\ref{b1.1}), we have in dimension four,
\[
\int_{\mathbb{B}^{4}}|\Delta
u|^{2}dx=\int_{\mathbb{B}^{4}}(-\Delta_{\mathbb{H}})(-\Delta_{\mathbb{H}}+2)u\cdot
udV\geq\frac{9}{16}\int_{\mathbb{B}^{4}}
u^{2}dV=9\int_{\mathbb{B}^{4}}\frac{u^{2}}{(1-|x|^{2})^{4}}dx.
\]
Furthermore,  the constant $9$  in above inequality is also sharp (see e.g. \cite{ow}).

\medskip

One of the main results of this paper is the following
\begin{theorem}\label{th1.1} Let $\alpha>0$.
Then there exists a constant $C_{\alpha}>0$ such that for all
$u\in C^{\infty}_{0}(\mathbb{B}^{4})$ with
\[
\int_{\mathbb{B}^{4}}\left(-\Delta_{\mathbb{H}}-9/4\right)(-\Delta_{\mathbb{H}}+\alpha)u\cdot
udV\leq1,
\]
there holds
\[
\int_{\mathbb{B}^{4}}(e^{32\pi^{2} u^{2}}-1-32\pi^{2} u^{2})dV=16\int_{\mathbb{B}^{4}}\frac{e^{32\pi^{2} u^{2}}-1-32\pi^{2} u^{2}}{(1-|x|^{2})^{4}}dx\leq C_{\alpha}.
\]
\end{theorem}

Choosing $\alpha=\frac{1}{4}$ and combing (\ref{b1.1}) and Theorem \ref{th1.1},  we have
 the following Hardy-Adams
inequalities
\begin{theorem}\label{c2}
There exists a constant $C_{1}>0$ such that for all $u\in
C^{\infty}_{0}(\mathbb{B}^{4})$ with
\[
\int_{\mathbb{B}^{4}}|\Delta
u|^{2}dx-9\int_{\mathbb{B}^{4}}\frac{u^{2}}{(1-|x|^{2})^{4}}dx\leq1,
\]
there holds
\[
\int_{\mathbb{B}^{4}}(e^{32\pi^{2}
u^{2}}-1-32u^{2})dV=16\int_{\mathbb{B}^{4}}\frac{e^{32\pi^{2}
u^{2}}-1-32u^{2}}{(1-|x|^{2})^{4}}dx\leq C_{1}.
\]
\end{theorem}

Theorem \ref{c2} implies the following improved  Adams
inequalities.
\begin{theorem}\label{th1.6} Let $\lambda<9$. Then
there exists a constant $C_{\lambda}>0$ such that for all $u\in
C^{\infty}_{0}(\mathbb{B}^{4})$ with
\[
\int_{\mathbb{B}^{4}}|\Delta
u|^{2}dx-\lambda\int_{\mathbb{B}^{4}}\frac{u^{2}}{(1-|x|^{2})^{4}}dx\leq1,
\]
there holds
\[
\int_{\mathbb{B}^{4}}(e^{32\pi^{2}
u^{2}}-1)dV=16\int_{\mathbb{B}^{4}}\frac{e^{32\pi^{2}
u^{2}}-1}{(1-|x|^{2})^{4}}dx\leq C_{\lambda}.
\]
\end{theorem}
As an application of the above theorem,
we also have the following Hardy-Adams inequality  which is a higher dimensional analogue of the
Hardy-Trudinger-Moser inequality given by Wang and Ye \cite{wy} and the authors \cite{ly}.
\begin{theorem} \label{th1.7}
There exists a constant $C_{3}>0$ such that for all $u\in
C^{\infty}_{0}(\mathbb{B}^{4})$ with
\[
\int_{\mathbb{B}^{4}}|\Delta u|^{2}dx-9\int_{\mathbb{B}^{4}}\frac{u^{2}}{(1-|x|^{2})^{4}}dx\leq1,
\]
there holds
\[
\int_{\mathbb{B}^{4}}e^{32\pi^{2} u^{2}}dx\leq C_{3}.
\]
\end{theorem}

Obviously, this theorem is stronger than the classical Adams inequality in \cite{ad} which is stated as:

\[
\int_{\mathbb{B}^{4}}e^{32\pi^{2} u^{2}}dx\leq C_{3}
\]
under the more restrictive constraint
$\int_{\mathbb{B}^{4}}|\Delta u|^{2}dx\le 1$
for all $u\in C_0^{\infty}(\mathbb{B}^4)$.

\medskip

We remark that in a forthcoming paper, we will establish the Hardy-Adams inequalities on hyperbolic spaces of all dimensions
$n$ when $n$ is even and $n\geq 4$.

\medskip

The organization of the paper is as follows: In Section 2, we review some necessary preliminaries on the hyperbolic spaces of Poincar\'e
model on the unit ball $\mathbb{B}^n$, the convolution, fractional Laplacian and Fourier transform  on the hyperbolic space $\mathbb{B}^n$ defined using the  the  Harish-Chandra $\mathfrak{c}$-function;   Section 3 gives the  pointwise estimates of Green function for fractional Laplacians; Section 4 establishes the symmetrization functions of the Green functions; Section 5 offers the proofs of our main results, namely Theorems \ref{th1.1}, \ref{c2}, \ref{th1.6} and \ref{th1.7}.

\section{Preliminaries on Fourier transform and fractional Laplacians on hyperbolic spaces }
We begin by quoting some preliminary facts which will be needed in
the sequel and  refer to \cite{ah,ge,he,he2,hu,liup} for more information about this subject.

 Recall that the Poincar\'e
model is the unit ball
\[\mathbb{B}^{n}=\{x=(x_{1},\cdots,x_{n})\in \mathbb{R}^{n}| |x|<1\}\]
equipped with the usual Poincar\'e metric
\[
ds^{2}=\frac{4(dx^{2}_{1}+\cdots+dx^{2}_{n})}{(1-|x|^{2})^{2}}.
\]
The distance from origin to $x\in \mathbb{B}^{n}$ is
\[
\rho(x)=\log\frac{1+|x|}{1-|x|}.
\]
and the polar coordinate is
\[
\int_{\mathbb{B}^{n}}fdV=\int^{+\infty}_{0}\int_{\mathbb{S}^{n-1}}f\cdot(\sinh\rho)^{n-1}
d\rho d\sigma,\;\; f\in L^{1}(\mathbb{B}^{n}).
\]

For each $a\in\mathbb{B}^{n}$, we define the  M\"obius
transformations $T_{a}$ by (see e.g. \cite{ah,hu})
\[
T_{a}(x)=\frac{|x-a|^{2}a-(1-|a|^{2})(x-a)}{1-2x\cdot
a+|x|^{2}|a|^{2}},
\]
where $x\cdot a=x_{1}a_{1}+x_{2}a_{2}+\cdots +x_{n}a_{n}$ denotes
the  scalar product in $\mathbb{R}^{n}$.
Using the M\"obius transformations,  the associated
distance  from $x$ to $y$ in $\mathbb{B}^{n}$ is
\begin{equation*}
\rho(x,y)=\rho(T_{x}(y))=\rho(T_{y}(x)).
\end{equation*}
Also using the M\"obius transformations,  we can define the convolution of
measurable functions $f$ and $g$ on $\mathbb{B}^{n}$ by (see e.g. \cite{liup})
\begin{equation}\label{2.1}
(f\ast g)(x)=\int_{\mathbb{B}^{n}}f(y)g(T_{x}(y))dV_{y}
\end{equation}
provided this integral exists. It is easy to check that
\[
f\ast g= g\ast f.
\]
Furthermore, if $g$ is radial, i.e. $g=g(\rho)$, then (see e.g. \cite{liup})
\begin{equation}\label{2.2}
  (f\ast g)\ast h= f\ast (g\ast h)
\end{equation}
provided $f,g,h\in L^{1}(\mathbb{B}^{n})$

Denote by $e^{t\Delta_{\mathbb{H}}}$ the heat kernel on $\mathbb{B}^{n}$. It is well known that $e^{t\Delta_{\mathbb{H}}}$ depends only on $t$ and $\rho(x,y)$. In fact, $e^{t\Delta_{\mathbb{H}}}$  is given explicitly by the following formulas (see e.g. \cite{d,gn}):

\begin{itemize}
  \item If $n=2m$, then
\[
e^{t\Delta_{\mathbb{H}}}=(2\pi)^{-\frac{n+1}{2}}t^{-\frac{1}{2}}e^{-\frac{(n-1)^{2}}{4}t}
\int^{+\infty}_{\rho}\frac{\sinh r}{\sqrt{\cosh r-\cosh\rho}}\left(-\frac{1}{\sinh r}\frac{\partial}{\partial r}\right)^{m}e^{-\frac{r^{2}}{4t}}dr;
\]
  \item If $n=2m+1$, then
  \[
  e^{t\Delta_{\mathbb{H}}}=2^{-m-1}\pi^{-m-1/2}t^{-\frac{1}{2}}e^{-\frac{(n-1)^{2}}{4}t}\left(-\frac{1}{\sinh \rho}\frac{\partial}{\partial \rho}\right)^{m}e^{-\frac{\rho^{2}}{4t}}.
  \]
\end{itemize}

 Finally, we review some basic facts about Fourier transform and fractional Laplacian  on hyperbolic space.
Set
\[
e_{\lambda,\zeta}(x)=\left(\frac{\sqrt{1-|x|^{2}}}{|x-\zeta|}\right)^{n-1+i\lambda}, \;\; x\in \mathbb{B}^{n},\;\;\lambda\in\mathbb{R},\;\;\zeta\in\mathbb{S}^{n-1}.
\]
The Fourier transform of a function  $f$  on $\mathbb{B}^{n}$ can be defined as
\[
\widehat{f}(\lambda,\zeta)=\int_{\mathbb{B}^{n}} f(x)e_{-\lambda,\zeta}(x)dV
\]
provided this integral exists. If $g\in C^{\infty}_{0}(\mathbb{B}^{n})$ is radial, then
 $$\widehat{(f\ast g)}=\widehat{f}\cdot\widehat{g}.$$
Moreover,
the following inversion formula holds for $f\in C^{\infty}_{0}(\mathbb{B}^{n})$ (see  \cite{liup}):
\[
f(x)=D_{n}\int^{+\infty}_{-\infty}\int_{\mathbb{S}^{n-1}} \widehat{f}(\lambda,\zeta)e_{\lambda,\zeta}(x)|\mathfrak{c}(\lambda)|^{-2}d\lambda d\sigma(\varsigma),
\]
where $D_{n}=\frac{1}{2^{3-n}\pi |\mathbb{S}^{n-1}|}$ and $\mathfrak{c}(\lambda)$ is the  Harish-Chandra $\mathfrak{c}$-function given by (see  \cite{liup})
\[
\mathfrak{c}(\lambda)=\frac{2^{n-1-i\lambda}\Gamma(n/2)\Gamma(i\lambda)}{\Gamma(\frac{n-1+i\lambda}{2})\Gamma(\frac{1+i\lambda}{2})}.
\]
Similarly, there holds the Plancherel formula:
\[
\int_{\mathbb{B}^{n}}|f(x)|^{2}dV=D_{n}\int^{+\infty}_{-\infty}\int_{\mathbb{S}^{n-1}}|\widehat{f}(\lambda,\zeta)|^{2}|\mathfrak{c}(\lambda)|^{-2}d\lambda d\sigma(\varsigma).
\]

Since $e_{\lambda,\zeta}(x)$ is an eigenfunction of $\Delta_{\mathbb{H}}$ with eigenvalue $-\frac{(n-1)^{2}+\lambda^{2}}{4}$, it is easy to check that, for
$f\in C^{\infty}_{0}(\mathbb{B}^{n})$,
\[
\widehat{\Delta_{\mathbb{H}}f}(\lambda,\zeta)=-\frac{(n-1)^{2}+\lambda^{2}}{4}\widehat{f}(\lambda,\zeta).
\]
Therefore, in analogy with the Euclidean setting, we define the fractional
Laplacian on hyperbolic space as follows:
\[
\widehat{(-\Delta_{\mathbb{H}})^{\gamma}f}(\lambda,\zeta)=\left(\frac{(n-1)^{2}+\lambda^{2}}{4}\right)^{\gamma}\widehat{f}(\lambda,\zeta),\;\;\gamma\in \mathbb{R}.
\]
For more information about fractional
Laplacian on hyperbolic space and and symmetric spaces, we refer to \cite{an,ba}.

\section{Sharp Estimates for the Green function and fractional power}
In the rest of paper, we shall fix $n=4$.
In what follows, $a \lesssim b$ will stand for $a\leq Cb$ with some positive absolute constant $C$.

Since the work of Adams \cite{ad}, it has been a standard approach to establish  sharp Trudinger-Moser and Adams inequalities in different settings including 
both Riemannian and sub-Riemannian settings such as on the Heisenberg groups  by using the sharp pointwise estimates of Green's functions 
together with using O'Neil's lemma of convolutions \cite{on}. We shall not go to the details here. To this end, we will derive the pointwise estimates for Green's functions of powers of Laplacians in the hyperbolic spaces to establish our Hardy-Adams inequalities. 

\medskip

Recall that
the heat kernel
\[
e^{t\Delta_{\mathbb{H}}}=(2\pi)^{-\frac{5}{2}}t^{-\frac{1}{2}}e^{-\frac{9}{4}t}
\int^{+\infty}_{\rho}\frac{\sinh r}{\sqrt{\cosh r-\cosh\rho}}\left(-\frac{1}{\sinh r}\frac{\partial}{\partial r}\right)^{2}e^{-\frac{r^{2}}{4t}}dr
\]
and the Mellin type expression on hyperbolic space (see e.g. \cite{anj}, Section 4.2)
$$(-\Delta_{\mathbb{H}}+\alpha)^{-\sigma}
=\frac{1}{\Gamma(\sigma)}\int^{+\infty}_{0}t^{\sigma-1}e^{t(\Delta_{\mathbb{H}}-\alpha)}dt,\;\;\;\;\alpha\geq-9/4,\;\;3/2>\sigma>0.$$
We have
\begin{equation*}
\begin{split}
(-\Delta_{\mathbb{H}}-9/4)^{-1}
=&\frac{1}{(\sqrt{2\pi})^{5}}\int^{+\infty}_{\rho}\frac{\sinh r}{\sqrt{\cosh r-\cosh\rho}}dr\int^{+\infty}_{0}t^{-\frac{1}{2}}\left(-\frac{1}{\sinh r}\frac{\partial}{\partial r}\right)^{2}e^{-\frac{r^{2}}{4t}}dt\\
=&\frac{1}{(\sqrt{2\pi})^{5}}\int^{+\infty}_{\rho}\frac{\sinh r}{\sqrt{\cosh r-\cosh\rho}}dr\left(-\frac{1}{\sinh r}\frac{\partial}{\partial r}\right)\int^{+\infty}_{0}\frac{r}{2\sinh r}t^{-\frac{3}{2}}e^{-\frac{r^{2}}{4t}}dt\\
=&\frac{1}{(\sqrt{2\pi})^{5}}\int^{+\infty}_{\rho}\frac{\sinh r}{\sqrt{\cosh r-\cosh\rho}}\left(-\frac{1}{\sinh r}\frac{\partial}{\partial r}\right)\frac{\Gamma(\frac{1}{2})}{\sinh r}dr\\
=&\frac{1}{4\sqrt{2}\pi^{2}}\int^{+\infty}_{\rho}\frac{\cosh r}{\sinh^{2}r\sqrt{\cosh r-\cosh\rho}}dr
\end{split}
\end{equation*}
Here we use the fact $\Gamma(1/2)=\sqrt{\pi}.$

\begin{lemma}\label{lm3.1}
There holds, for $\rho>0$,
\[
(-\Delta_{\mathbb{H}}-9/4)^{-1}\leq\frac{1}{4\pi^{2}\cosh\frac{\rho}{2}\sinh^{2}\rho}+\frac{1}{4\pi^{2}\cosh\frac{\rho}{2}\sinh\rho}.
\]
\end{lemma}

\begin{proof}
Using the substitution $t=\sqrt{\cosh r-\cosh\rho}$, we have
\begin{equation*}
\begin{split}
(-\Delta_{\mathbb{H}}-9/4)^{-1}
=&\frac{1}{4\sqrt{2}\pi^{2}}\int^{+\infty}_{\rho}\frac{\cosh r}{\sinh^{2}r\sqrt{\cosh r-\cosh\rho}}dr\\
=&\frac{1}{2\sqrt{2}\pi^{2}}\int^{+\infty}_{0}\frac{t^{2}+\cosh \rho}{[(t^{2}+\cosh \rho)^{2}-1]^{\frac{3}{2}}}dt\\
=&\frac{1}{2\sqrt{2}\pi^{2}}\int^{+\infty}_{0}\frac{t^{2}+\cosh \rho-1}{[(t^{2}+\cosh \rho)^{2}-1]^{\frac{3}{2}}}dt+\frac{1}{2\sqrt{2}\pi^{2}}\int^{+\infty}_{0}\frac{1}{[(t^{2}+\cosh \rho)^{2}-1]^{\frac{3}{2}}}dt\\
=&:(I)+(II)
\end{split}
\end{equation*}
where
\begin{equation*}
\begin{split}
(I)=&\frac{1}{2\sqrt{2}\pi^{2}}\int^{+\infty}_{0}\frac{t^{2}+\cosh \rho-1}{[(t^{2}+\cosh \rho)^{2}-1]^{\frac{3}{2}}}dt\\
=&\frac{1}{2\sqrt{2}\pi^{2}}\int^{+\infty}_{0}\frac{1}{(t^{2}+\cosh \rho+1)^{\frac{3}{2}}(t^{2}+\cosh \rho-1)^{\frac{1}{2}}}dt\\
\leq&\frac{1}{2\sqrt{2}\pi^{2}}\frac{1}{\sqrt{\cosh \rho-1}}\int^{+\infty}_{0}\frac{1}{(t^{2}+\cosh \rho+1)^{\frac{3}{2}}}dt\\
=&\frac{1}{2\sqrt{2}\pi^{2}}\frac{1}{\sqrt{\cosh \rho-1}}\cdot\left.\frac{1}{\cosh \rho+1}\frac{t}{\sqrt{t^{2}+\cosh \rho+1}}\right|^{\infty}_{0}\\
=&\frac{1}{2\sqrt{2}\pi^{2}}\frac{1}{\sinh\rho\sqrt{\cosh \rho+1}}=\frac{1}{4\pi^{2}\cosh\frac{\rho}{2}\sinh\rho};
\end{split}
\end{equation*}
\begin{equation*}
\begin{split}
(II)=&\frac{1}{2\sqrt{2}\pi^{2}}\int^{+\infty}_{0}\frac{1}{(t^{2}+\cosh \rho+1)^{\frac{3}{2}}(t^{2}+\cosh \rho-1)^{\frac{3}{2}}}dt\\
\leq&\frac{1}{2\sqrt{2}\pi^{2}}\frac{1}{(\cosh \rho+1)\frac{3}{2}}\int^{+\infty}_{0}\frac{1}{(t^{2}+\cosh \rho-1)^{\frac{3}{2}}}dt\\
=&\frac{1}{2\sqrt{2}\pi^{2}}\frac{1}{(\cosh \rho+1)\frac{3}{2}}\cdot\left.\frac{1}{\cosh \rho-1}\frac{t}{\sqrt{t^{2}+\cosh \rho-1}}\right|^{\infty}_{0}\\
=&\frac{1}{2\sqrt{2}\pi^{2}}\frac{1}{\sinh^{2}\rho\sqrt{\cosh \rho+1}}=\frac{1}{4\pi^{2}\cosh\frac{\rho}{2}\sinh^{2}\rho}.
\end{split}
\end{equation*}
The desired result follows.
\end{proof}

Also via the heat kernel and  the Mellin type expression, the fractional power
\begin{equation}\label{b3.2}
\begin{split}
&(-\Delta_{\mathbb{H}}+\alpha)^{-\frac{1}{2}}=\frac{1}{\Gamma(1/2)}\int^{+\infty}_{0}t^{-\frac{1}{2}}e^{t(\Delta_{\mathbb{H}}-\alpha)}dt\\
=&\frac{1}{\sqrt{\pi}(\sqrt{2\pi})^{5}}\int^{+\infty}_{\rho}\frac{\sinh
r}{\sqrt{\cosh r-\cosh\rho}}dr\int^{+\infty}_{0}t^{-1}e^{-\alpha
t-\frac{9}{4}t}\left(-\frac{1}{\sinh r}\frac{\partial}{\partial
r}\right)^{2}e^{-\frac{r^{2}}{4t}}dt,
\end{split}
\end{equation}
where  $\alpha\geq-9/4$. It is easy to check that if $\alpha>-9/4$, then
\begin{equation}\label{3.1}
\begin{split}
(-\Delta_{\mathbb{H}}+\alpha)^{-\frac{1}{2}}\leq(-\Delta_{\mathbb{H}}-9/4)^{-\frac{1}{2}}.
\end{split}
\end{equation}
Furthermore, we have the following estimates of $(-\Delta_{\mathbb{H}}-9/4)^{-\frac{1}{2}}$.

\begin{lemma} \label{lm3.2} There holds, for $\rho>0$
\begin{equation*}
\begin{split}
(-\Delta_{\mathbb{H}}-9/4)^{-\frac{1}{2}}\leq&\frac{1}{16\pi^{2}(1+\cosh\rho)}\cdot\frac{1}{\sinh^{3}\frac{\rho}{2}}+\frac{\sqrt{2}}{4\pi^{2}\sqrt{1+\cosh\rho}}\cdot\frac{1}{\sinh\rho};\\
(-\Delta_{\mathbb{H}}-9/4)^{-\frac{1}{2}}\leq&\frac{8}{\sqrt{\pi}(\sqrt{2\pi})^{5}}\cdot\frac{1}{\rho\sqrt{\cosh\rho+1}(\cosh\rho-1)}.
\end{split}
\end{equation*}
\end{lemma}
\begin{proof}
By (\ref{b3.2}),
\[
(-\Delta_{\mathbb{H}}-9/4)^{-\frac{1}{2}}=\frac{1}{\sqrt{\pi}(\sqrt{2\pi})^{5}}\int^{+\infty}_{\rho}\frac{\sinh
r}{\sqrt{\cosh r-\cosh\rho}}dr\int^{+\infty}_{0}t^{-1}\left(-\frac{1}{\sinh r}\frac{\partial}{\partial
r}\right)^{2}e^{-\frac{r^{2}}{4t}}dt.
\]
Notice that, for $r>0$,
\begin{equation*}
\begin{split}
\int^{+\infty}_{0}t^{-1}\left(-\frac{1}{\sinh r}\frac{\partial}{\partial r}\right)^{2}e^{-\frac{r^{2}}{4t}}dt
=&\int^{+\infty}_{0}\left(-\frac{1}{\sinh r}\frac{\partial}{\partial r}\right)\left(\frac{r}{\sinh r}\cdot\frac{1}{2t^{2}}e^{-\frac{r^{2}}{4t}}\right)dt\\
=&\left(-\frac{1}{\sinh r}\frac{\partial}{\partial r}\right)\left(\frac{r}{2\sinh r}\cdot\int^{+\infty}_{0}t^{-2}e^{-\frac{r^{2}}{4t}}dt\right)\\
=&\left(-\frac{1}{\sinh r}\frac{\partial}{\partial r}\right)\frac{2}{r\sinh r}=\frac{2}{r^{2}\sinh^{2} r}+2\frac{\cosh r}{r\sinh^{3}r},
\end{split}
\end{equation*}
we have,
\begin{equation}\label{b3.1}
\begin{split}
(-\Delta_{\mathbb{H}}-9/4)^{-\frac{1}{2}}
=&\frac{2}{\sqrt{\pi}(\sqrt{2\pi})^{5}}\int^{+\infty}_{\rho}\frac{1}{\sqrt{\cosh r-\cosh\rho}}\left(\frac{1}{r^{2}\sinh r}+\frac{\cosh r}{r\sinh^{2}r}\right)dr\\
\leq&\frac{2}{\sqrt{\pi}(\sqrt{2\pi})^{5}}\int^{+\infty}_{\rho}\frac{1}{\sqrt{\cosh r-\cosh\rho}}\left(\frac{\cosh r}{r\sinh^{2}r}+\frac{\cosh r}{r\sinh^{2}r}\right)dr\\
=&\frac{4}{\sqrt{\pi}(\sqrt{2\pi})^{5}}\int^{+\infty}_{\rho}\frac{\cosh r}{r\sinh^{2}r\sqrt{\cosh r-\cosh\rho}}dr.
\end{split}
\end{equation}
To get the first inequality in (\ref{b3.1}), we use the inequality $\frac{1}{r}\leq \coth r (r>0)$. Therefore,
\begin{equation*}
\begin{split}
(-\Delta_{\mathbb{H}}-9/4)^{-\frac{1}{2}}
\leq&\frac{4}{\sqrt{\pi}(\sqrt{2\pi})^{5}}\int^{+\infty}_{\rho}\frac{\cosh r}{r\sinh^{2}r\sqrt{\cosh r-\cosh\rho}}dr\\
\leq&\frac{2}{\sqrt{\pi}(\sqrt{2\pi})^{5}}\int^{+\infty}_{\rho}\frac{\cosh r(\cosh r+1)}{\sinh^{3}r\sqrt{\cosh r-\cosh\rho}}dr.
\end{split}
\end{equation*}
Here we use the fact $\frac{1}{r}\leq \frac{1}{2}\coth\frac{r}{2}=\frac{1+\cosh r}{2\sinh r} (r>0)$.
Using the substitution $t=\sqrt{\cosh r-\cosh\rho}$ yields
\begin{equation*}
\begin{split}
(-\Delta_{\mathbb{H}}-9/4)^{-\frac{1}{2}}\leq&\frac{2}{\sqrt{\pi}(\sqrt{2\pi})^{5}}\int^{+\infty}_{\rho}\frac{\cosh r(\cosh r+1)}{\sinh^{3}r\sqrt{\cosh r-\cosh\rho}}dr\\
=&\frac{2}{\sqrt{\pi}(\sqrt{2\pi})^{5}}\int^{+\infty}_{\rho}\frac{\cosh r}{\sinh r(\cosh r-1)\sqrt{\cosh r-\cosh\rho}}dr\\
=&\frac{4}{\sqrt{\pi}(\sqrt{2\pi})^{5}}\int^{+\infty}_{0}\frac{t^{2}+\cosh\rho}{(t^{2}+\cosh\rho+1)(t^{2}+\cosh\rho -1)^{2}}dt\\
=&\frac{4}{\sqrt{\pi}(\sqrt{2\pi})^{5}}\int^{+\infty}_{0}\frac{t^{2}+\cosh\rho-1}{(t^{2}+\cosh\rho+1)(t^{2}+\cosh\rho -1)^{2}}dt\\
&+\frac{4}{\sqrt{\pi}(\sqrt{2\pi})^{5}}\int^{+\infty}_{0}\frac{1}{(t^{2}+\cosh\rho+1)(t^{2}+\cosh\rho -1)^{2}}dt\\
=&:(III)+(IV),
\end{split}
\end{equation*}
where
\begin{equation*}
\begin{split}
(III)=&\frac{4}{\sqrt{\pi}(\sqrt{2\pi})^{5}}\int^{+\infty}_{0}\frac{t^{2}+\cosh\rho-1}{(t^{2}+\cosh\rho+1)(t^{2}+\cosh\rho -1)^{2}}dt\\
=&\frac{4}{\sqrt{\pi}(\sqrt{2\pi})^{5}}\int^{+\infty}_{0}\frac{1}{(t^{2}+\cosh\rho+1)(t^{2}+\cosh\rho -1)}dt\\
\leq&\frac{4}{\sqrt{\pi}(\sqrt{2\pi})^{5}}\frac{1}{1+\cosh\rho}\cdot\int^{+\infty}_{0}\frac{1}{t^{2}+\cosh\rho -1}dt\\
=&\frac{4}{\sqrt{\pi}(\sqrt{2\pi})^{5}}\frac{1}{1+\cosh\rho}\cdot\frac{1}{\sqrt{\cosh\rho-1}}\frac{\pi}{2}
\\=&\frac{\sqrt{2}}{4\pi^{2}\sqrt{1+\cosh\rho}}\cdot\frac{1}{\sinh\rho};
\\
(IV)=&\frac{4}{\sqrt{\pi}(\sqrt{2\pi})^{5}}\int^{+\infty}_{0}\frac{1}{(t^{2}+\cosh\rho+1)(t^{2}+\cosh\rho -1)^{2}}dt\\
\leq&\frac{4}{\sqrt{\pi}(\sqrt{2\pi})^{5}}\frac{1}{(1+\cosh\rho)}\int^{+\infty}_{0}\frac{1}{(t^{2}+\cosh\rho -1)^{2}}dt\\
=&\frac{4}{\sqrt{\pi}(\sqrt{2\pi})^{5}}\frac{1}{(1+\cosh\rho)}\frac{1}{(\cosh\rho-1)^{3/2}} \frac{\pi}{4}\\
=&\frac{1}{16\pi^{2}(1+\cosh\rho)}\cdot\frac{1}{\sinh^{3}\frac{\rho}{2}}.
\end{split}
\end{equation*}
Thus,
\begin{equation*}
\begin{split}
(-\Delta_{\mathbb{H}}-9/4)^{-\frac{1}{2}}=&(III)+(IV)\leq\frac{1}{16\pi^{2}(1+\cosh\rho)}\cdot\frac{1}{\sinh^{3}\frac{\rho}{2}}+\frac{\sqrt{2}}{4\pi^{2}\sqrt{1+\cosh\rho}}\cdot\frac{1}{\sinh\rho}.
\end{split}
\end{equation*}

On the other hand, for $\rho>0$, we have, by (\ref{b3.1}),
\begin{equation*}
\begin{split}
(-\Delta_{\mathbb{H}}-9/4)^{-\frac{1}{2}}
\leq&\frac{4}{\sqrt{\pi}(\sqrt{2\pi})^{5}}\int^{+\infty}_{\rho}\frac{\cosh r}{r\sinh^{2}r\sqrt{\cosh r-\cosh\rho}}dr\\
\leq&\frac{4}{\sqrt{\pi}(\sqrt{2\pi})^{5}\rho}\int^{+\infty}_{\rho}\frac{\cosh r}{\sinh^{2}r\sqrt{\cosh r-\cosh\rho}}dr.
\end{split}
\end{equation*}
Also using the substitution $t=\sqrt{\cosh r-\cosh\rho}$, we have
\begin{equation*}
\begin{split}
(-\Delta_{\mathbb{H}}-9/4)^{-\frac{1}{2}}
\leq&\frac{8}{\sqrt{\pi}(\sqrt{2\pi})^{5}\rho}\int^{+\infty}_{0}\frac{t^{2}+\cosh \rho}{[(t^{2}+\cosh \rho)^{2}-1]^{\frac{3}{2}}}dt\\
\leq&\frac{8}{\sqrt{\pi}(\sqrt{2\pi})^{5}\rho}\int^{+\infty}_{0}\frac{1}{(t^{2}+\cosh \rho-1)^{\frac{1}{2}}(t^{2}+\cosh \rho-1)^{\frac{3}{2}}}dt\\
\leq&\frac{8}{\sqrt{\pi}(\sqrt{2\pi})^{5}\rho}\cdot\frac{1}{\sqrt{\cosh\rho+1}}\int^{+\infty}_{0}\frac{1}{(t^{2}+\cosh \rho-1)^{\frac{3}{2}}}dt\\
=&\frac{8}{\sqrt{\pi}(\sqrt{2\pi})^{5}}\cdot\frac{1}{\rho\sqrt{\cosh\rho+1}(\cosh\rho-1)}.
\end{split}
\end{equation*}
The proof of Lemma \ref{lm3.2} is then completed.
\end{proof}

\begin{corollary} \label{c1} There holds, for $\rho>0$,
\begin{equation*}
\begin{split}
(-\Delta_{\mathbb{H}}-9/4)^{-\frac{1}{2}}\leq&\frac{1}{4\pi^{2}\sinh^{3}\rho}+\left(\frac{1}{8\pi^{2}\cosh\rho\cosh\rho/2}+\frac{\sqrt{2}}{4\pi^{2}\sqrt{1+\cosh\rho}}\right)\frac{1}{\sinh\rho}
\end{split}
\end{equation*}
and
\begin{equation*}
\begin{split}
(-\Delta_{\mathbb{H}}-9/4)^{-\frac{1}{2}}\lesssim\rho^{-1}e^{-\frac{3}{2}\rho},\;\;\;\;\rho>1.
\end{split}
\end{equation*}
\end{corollary}
\begin{proof}
By Lemma \ref{lm3.2},
\begin{equation*}
\begin{split}
(-\Delta_{\mathbb{H}}-9/4)^{-\frac{1}{2}}\leq&\frac{1}{16\pi^{2}(1+\cosh\rho)}\cdot\frac{1}{\sinh^{3}\frac{\rho}{2}}
\cdot\left(\frac{1}{\cosh\rho}+\frac{\cosh\rho-1}{\cosh\rho}\right)
+\frac{\sqrt{2}}{4\pi^{2}\sqrt{1+\cosh\rho}}\cdot\frac{1}{\sinh\rho}\\
=&\frac{1}{16\pi^{2}\cosh\rho(1+\cosh\rho)}\cdot\frac{1}{\sinh^{3}\frac{\rho}{2}}+\frac{\cosh\rho-1}{16\pi^{2}\cosh\rho(1+\cosh\rho)}\cdot\frac{1}{\sinh^{3}\frac{\rho}{2}}+\\
&\frac{\sqrt{2}}{4\pi^{2}\sqrt{1+\cosh\rho}}\cdot\frac{1}{\sinh\rho}\\
=&\frac{1}{16\pi^{2}\cosh\rho(1+\cosh\rho)}\cdot\frac{1}{\sinh^{3}\frac{\rho}{2}}+\frac{1}{8\pi^{2}\cosh\rho(1+\cosh\rho)}\cdot\frac{1}{\sinh\frac{\rho}{2}}+\\
&\frac{\sqrt{2}}{4\pi^{2}\sqrt{1+\cosh\rho}}\cdot\frac{1}{\sinh\rho}.
\end{split}
\end{equation*}
Since $1+\cosh\rho=2\cosh^{2}\rho/2$ and $\cosh\rho/2\leq\cosh\rho$, we have
\begin{equation*}
\begin{split}
(-\Delta_{\mathbb{H}}-9/4)^{-\frac{1}{2}}
\leq&\frac{1}{16\pi^{2}\cosh\rho/2\cdot2\cosh^{2}\rho/2}\cdot\frac{1}{\sinh^{3}\frac{\rho}{2}}+\frac{1}{8\pi^{2}\cosh\rho \cdot2\cosh^{2}\rho/2}\cdot\frac{1}{\sinh\frac{\rho}{2}}+\\
&\frac{\sqrt{2}}{4\pi^{2}\sqrt{1+\cosh\rho}}\cdot\frac{1}{\sinh\rho}\\
=&\frac{1}{4\pi^{2}\sinh^{3}\rho}+\left(\frac{1}{8\pi^{2}\cosh\rho\cosh\rho/2}+\frac{\sqrt{2}}{4\pi^{2}\sqrt{1+\cosh\rho}}\right)\frac{1}{\sinh\rho}.
\end{split}
\end{equation*}
Also by Lemma \ref{lm3.2},
\begin{equation*}
\begin{split}
(-\Delta_{\mathbb{H}}-9/4)^{-\frac{1}{2}}\leq&\frac{8}{\sqrt{\pi}(\sqrt{2\pi})^{5}}\cdot\frac{1}{\rho\sqrt{\cosh\rho+1}(\cosh\rho-1)}\\
\lesssim&\frac{1}{\rho e^{\frac{3}{2}\rho}},\;\;\;\rho>1.
\end{split}
\end{equation*}
This completes the proof of Corollary \ref{c1}.
\end{proof}

\begin{lemma} \label{lm3.3} Let $\alpha>0$. Then there exist   $\alpha_{0}>0$  such that
\[
(-\Delta_{\mathbb{H}}+\alpha)^{-\frac{1}{2}}\lesssim
e^{-(3+\alpha_{0})\rho},\;\;\;\;\;\rho>1.
\]
\end{lemma}
\begin{proof} We have
\begin{equation*}
\begin{split}
(-\Delta_{\mathbb{H}}+\alpha)^{-\frac{1}{2}}=&\frac{1}{\sqrt{\pi}(\sqrt{2\pi})^{5}}\int^{+\infty}_{\rho}\frac{\sinh r}{\sqrt{\cosh r-\cosh\rho}}dr\int^{+\infty}_{0}t^{-1}e^{-\alpha t-\frac{9}{4}t}\left(-\frac{1}{\sinh r}\frac{\partial}{\partial r}\right)^{2}e^{-\frac{r^{2}}{4t}}dt\\
=&\frac{1}{2\sqrt{\pi}(\sqrt{2\pi})^{5}}\int^{+\infty}_{\rho}\frac{r\cosh r-\sinh r}{\sinh^{2}r\sqrt{\cosh r-\cosh\rho}}dr\int^{+\infty}_{0}t^{-2}e^{-\alpha t-\frac{9}{4}t}e^{-\frac{r^{2}}{4t}}dt+\\
&\frac{1}{4\sqrt{\pi}(\sqrt{2\pi})^{5}}\int^{+\infty}_{\rho}\frac{r^{2}}{\sinh r\sqrt{\cosh r-\cosh\rho}}dr\int^{+\infty}_{0}t^{-3}e^{-\alpha t-\frac{9}{4}t}e^{-\frac{r^{2}}{4t}}dt\\
=&:(V)+(VI),
\end{split}
\end{equation*}
where
\begin{equation*}
\begin{split}
(V)=&\frac{1}{2\sqrt{\pi}(\sqrt{2\pi})^{5}}\int^{+\infty}_{\rho}\frac{r\cosh r-\sinh r}{\sinh^{2}r\sqrt{\cosh r-\cosh\rho}}dr\int^{+\infty}_{0}t^{-2}e^{-\alpha t-\frac{9}{4}t}e^{-\frac{r^{2}}{4t}}dt\\
 \lesssim&\int^{+\infty}_{\rho}\frac{r\cosh r}{\sinh^{2}r\sqrt{\cosh r-\cosh\rho}}dr\int^{+\infty}_{0}t^{-2}e^{-\alpha t-\frac{9}{4}t}e^{-\frac{r^{2}}{4t}}dt;\\
(VI)=&\frac{1}{4\sqrt{\pi}(\sqrt{2\pi})^{5}}\int^{+\infty}_{\rho}\frac{r^{2}}{\sinh
r\sqrt{\cosh r-\cosh\rho}}dr\int^{+\infty}_{0}t^{-3}e^{-\alpha
t-\frac{9}{4}t}e^{-\frac{r^{2}}{4t}}dt.
\end{split}
\end{equation*}
Notice that $\forall \epsilon\in (0,1)$, $e^{-\alpha t-\frac{9}{4}t}e^{-\frac{(1-\epsilon)r^{2}}{4t}}$ has as  a maximum value $e^{-\sqrt{(\alpha+\frac{9}{4})(1-\epsilon)}r}$. We have,
\begin{equation*}
\begin{split}
\int^{+\infty}_{0}t^{-2}e^{-\alpha
t-\frac{9}{4}t}e^{-\frac{r^{2}}{4t}}dt\leq&
e^{-\sqrt{(\alpha+\frac{9}{4})(1-\epsilon)}r}\int^{+\infty}_{0}t^{-2}e^{-\frac{\epsilon
r^{2}}{4t}}dt=e^{-\sqrt{(\alpha+\frac{9}{4})(1-\epsilon)}r}\cdot\frac{4}{\epsilon r^{2}};\\
\int^{+\infty}_{0}t^{-3}e^{-\alpha
t-\frac{9}{4}t}e^{-\frac{r^{2}}{4t}}dt\leq&
e^{-\sqrt{(\alpha+\frac{9}{4})(1-\epsilon)}r}\int^{+\infty}_{0}t^{-3}e^{-\frac{\epsilon
r^{2}}{4t}}dt=e^{-\sqrt{(\alpha+\frac{9}{4})(1-\epsilon)}r}\cdot\frac{16}{\epsilon
r^{4}}.
\end{split}
\end{equation*}
Choose $\epsilon_{0}\in (0,1)$ such that
$
\alpha_{0}=\sqrt{(\alpha+\frac{9}{4})(1-\epsilon_{0})}-\frac{3}{2}>0.
$
Then, for $\rho>1$,
\begin{equation*}
\begin{split}
(V)\lesssim&\int^{+\infty}_{\rho}\frac{r\cosh r}{\sinh^{2}r\sqrt{\cosh r-\cosh\rho}}\cdot\frac{1}{r^{2}e^{\frac{3}{2}r+\alpha_{0}r}}dr\\
\lesssim&\int^{+\infty}_{\rho}\frac{1}{\sinh^{2}r\sqrt{\cosh r-\cosh\rho}}\cdot\frac{1}{re^{\frac{1}{2}r+\alpha_{0}r}}dr\\
\lesssim&\frac{1}{\rho e^{\frac{1}{2}\rho+\alpha_{0}\rho}}\int^{+\infty}_{\rho}\frac{1}{\sinh^{2}r\sqrt{\cosh r-\cosh\rho}}dr\\
\leq&\frac{1}{ e^{\frac{1}{2}\rho+\alpha_{0}\rho}}\int^{+\infty}_{\rho}\frac{1}{\sinh^{2}r\sqrt{\cosh r-\cosh\rho}}dr\\
\leq&\frac{1}{ e^{\frac{1}{2}\rho+\alpha_{0}\rho}\sinh\rho}\int^{+\infty}_{\rho}\frac{1}{\sinh r\sqrt{\cosh r-\cosh\rho}}dr;
\end{split}
\end{equation*}
\begin{equation*}
\begin{split}
(VI)\lesssim&\int^{+\infty}_{\rho}\frac{1}{\sinh
r\sqrt{\cosh r-\cosh\rho}}\cdot
\frac{1}{r^{2}e^{\frac{3}{2}r+\alpha_{0}r}}dr\\
\lesssim&
\frac{1}{e^{\frac{3}{2}\rho+\alpha_{0}\rho}}\int^{+\infty}_{\rho}\frac{1}{\sinh r\sqrt{\cosh r-\cosh\rho}}dr.
\end{split}
\end{equation*}
On the other hand, using the substitution $t=\sqrt{\cosh
r-\cosh\rho}$, we have
\begin{equation*}
\begin{split}
\int^{+\infty}_{\rho}\frac{1}{\sinh r\sqrt{\cosh r-\cosh\rho}}dr=&2\int^{+\infty}_{0}\frac{1}{(t^{2}+\cosh \rho)^{2}-1}dt\\
=&2\int^{+\infty}_{0}\frac{1}{(t^{2}+\cosh \rho+1)(t^{2}+\cosh \rho-1)}dt\\
\leq&\frac{2}{\cosh\rho+1}\int^{+\infty}_{0}\frac{1}{t^{2}+\cosh \rho-1}dt\\
=&\frac{2}{\cosh\rho+1}\cdot\frac{1}{\sqrt{\cosh\rho-1}}\frac{\pi}{2}.
\end{split}
\end{equation*}
Therefore, for $\rho>1$,
\begin{equation*}
\begin{split}
(V)\lesssim&\frac{1}{ e^{\frac{1}{2}\rho+\alpha_{0}\rho}\sinh\rho}\cdot\frac{2}{\cosh\rho+1}\cdot\frac{1}{\sqrt{\cosh\rho-1}}\frac{\pi}{2}\lesssim\frac{1}{e^{(3+\alpha_{0})\rho} };\\
(VI)\lesssim&
\frac{1}{e^{\frac{3}{2}\rho+\alpha_{0}\rho}}\cdot\frac{2}{\cosh\rho+1}\cdot\frac{1}{\sqrt{\cosh\rho-1}}\frac{\pi}{2}\lesssim\frac{1}{e^{(3+\alpha_{0})\rho} }.
\end{split}
\end{equation*}
Thus
\[
(-\Delta_{\mathbb{H}}+\alpha)^{-\frac{1}{2}}=(V)+(VI)\lesssim\frac{1}{e^{(3+\alpha_{0})\rho} },\;\;\;\;\;\rho>1.
\]
The proof is then completed.
\end{proof}

\section{rearrangement}
We now recall the rearrangement of a real functions on $\mathbb{B}^{4}$.  Suppose $f$ is
a real  function on $\mathbb{B}^{4}$. The non-increasing rearrangement of $f$
is defined by
\begin{equation*}
f^{\ast}(t)=\inf\{s>0: \lambda_{f}(s)\leq t\},
\end{equation*}
where $$\lambda_{f}(s)=|\{x\in \mathbb{B}^{4}: |f(x)|>s\}|=\int_{\{x\in \mathbb{B}^{4}:: |f(x)|>s\}}\left(\frac{2}{1-|x|^{2}}\right)^{4}dx.$$
 Here we use the
notation $|\Sigma|$ for the measure of a measurable set
$\Sigma\subset \mathbb{B}^{4}$.

\begin{lemma}\label{lm4.1}
There exists a constant $A_{1}>0$ such that
 \[
[(-9/4-\Delta_{\mathbb{H}})^{-1}]^{\ast}(t)\leq\frac{1}{4\sqrt{2}\pi\sqrt{t}}(1+A_{1}t^{\frac{1}{4}}),\;\;\;t>0.
\]
\end{lemma}
\begin{proof} Set $\phi(\rho)=(-9/4-\Delta_{\mathbb{H}})^{-1}$.
Define, for any $s>0$,
\begin{equation}\label{4.1}
\lambda_{\phi}(s)=\int_{\{\phi(\rho)>s\}}dV=|\mathbb{S}^{3}|\int^{\rho_{s}}_{0}\sinh^{3}\rho
d\rho=2\pi^{2}\int^{\rho_{s}}_{0}\sinh^{3}\rho d\rho,
\end{equation}
where $\rho_{s}$ is the solution of equation
\begin{equation}\label{4.2}
\phi(\rho)=s.
\end{equation}
Therefore, since $\phi^{\ast}(t)=\inf\{s>0: \lambda_{\phi}(s)\leq t\}$, we have
\begin{equation}\label{4.3}
t=\lambda_{\phi}(\phi^{\ast}(t))=2\pi^{2}\int^{\rho_{\phi^{\ast}(t)}}_{0}\sinh^{3}\rho
d\rho,
\end{equation}
where $\rho_{g^{\ast}(t)}$ satisfies
\begin{equation}\label{4.4}
\phi( \rho_{\phi^{\ast}(t)})=\phi^{\ast}(t).
\end{equation}
Notice that
\begin{equation}\label{b4.5}
t=2\pi^{2}\int^{\rho_{\phi^{\ast}(t)}}_{0}\sinh^{3}\rho d\rho\leq2\pi^{2}\int^{\rho_{\phi^{\ast}(t)}}_{0}\sinh^{3}\rho\cosh\rho d\rho=\frac{\pi^{2}}{2}\sinh^{4}\rho_{\phi^{\ast}(t)};
\end{equation}
\begin{equation}\label{b4.6}
t=2\pi^{2}\int^{\rho_{\phi^{\ast}(t)}}_{0}\sinh^{3}\rho d\rho\leq2\pi^{2}\int^{\rho_{\phi^{\ast}(t)}}_{0}e^{3\rho} d\rho=\frac{2\pi^{2}}{3}e^{3\rho_{\phi^{\ast}(t)}}.
\end{equation}
We have, by (\ref{4.4})-(\ref{b4.5}) and Lemma 3.1,
\begin{equation*}
\begin{split}
\phi^{\ast}(t)\leq&\frac{1}{4\sqrt{2}\pi\sqrt{t}\cdot\cosh\frac{\rho_{\phi^{\ast}(t)}}{2}}+
\frac{1}{4\pi^{2}\cosh\frac{\rho_{\phi^{\ast}(t)}}{2}}\left(\frac{\pi^{2}}{2t}\right)^{\frac{1}{4}}\\
\leq&\frac{1}{4\sqrt{2}\pi\sqrt{t}}(1+A_{1}t^{\frac{1}{4}}),
\end{split}
\end{equation*}
where $A_{1}=2^{\frac{1}{4}}\pi^{-\frac{1}{2}}$.
The desired result follows.
\end{proof}

Similarly, by Corollary \ref{c1} and Lemma \ref{lm3.3}, we have the following (we omit the proof because it is completely the same to that of Lemma \ref{lm4.1}).
\begin{lemma} \label{lm4.2}Let $\alpha>0$ and $\epsilon_{0}$ be in Lemma 3.3.
Then there exists a constant $A_{2}>0$ such that
\begin{equation*}
\begin{split}
[(-\Delta_{\mathbb{H}}-9/4)^{-\frac{1}{2}}]^{\ast}(t)\leq&\frac{2^{\frac{1}{4}}}{8\sqrt{\pi}t^{\frac{3}{4}}}(1+A_{2}\sqrt{t}),\;\;\;t>0;\\
[(-\Delta_{\mathbb{H}}-9/4)^{-\frac{1}{2}}]^{\ast}(t)\lesssim&\frac{1}{
\sqrt{t}\ln t},\;\;\;t>2;\\
[(-\Delta_{\mathbb{H}}+\alpha)^{-\frac{1}{2}}]^{\ast}(t)\lesssim&\frac{1}{t^{1+\frac{1}{3}\epsilon_{0}}},\;\;\;t>2.
\end{split}
\end{equation*}
\end{lemma}

Since for $\alpha>-9/4$, \[
(-\Delta_{\mathbb{H}}+\alpha)^{-\frac{1}{2}}\leq(-\Delta_{\mathbb{H}}-9/4)^{-\frac{1}{2}},
\]
we  have, by Lemma \ref{lm4.2},
\begin{equation}\label{4.5}
\begin{split}
[(-\Delta_{\mathbb{H}}+\alpha)^{-\frac{1}{2}}]^{\ast}(t)\leq[(-\Delta_{\mathbb{H}}-9/4)^{-\frac{1}{2}}]^{\ast}(t)\leq\frac{2^{\frac{1}{4}}}{8\sqrt{\pi}t^{\frac{3}{4}}}(1+A_{2}\sqrt{t}),\;\;\;t>0.
\end{split}
\end{equation}

\begin{lemma} \label{lm4.3}Let $\alpha>0$ and $\epsilon_{0}$ be in Lemma 3.3.
Then
\begin{equation}\label{4.6}
\begin{split}
[(-\Delta_{\mathbb{H}}-9/4)^{-\frac{1}{2}}\ast(-\Delta_{\mathbb{H}}+\alpha)^{-\frac{1}{2}}]^{\ast}(t)\leq\frac{1}{4\sqrt{2}\pi\sqrt{t}}(1+A_{1}t^{\frac{1}{4}}),\;\;\;t>0
\end{split}
\end{equation}
and
\begin{equation}\label{4.7}
\begin{split}
[(-\Delta_{\mathbb{H}}-9/4)^{-\frac{1}{2}}\ast(-\Delta_{\mathbb{H}}+\alpha)^{-\frac{1}{2}}]^{\ast}(t)\lesssim\frac{1}{\sqrt{t}
\ln t},\;\;t>2.
\end{split}
\end{equation}
\end{lemma}
\begin{proof}
Since
$(-\Delta_{\mathbb{H}}-9/4)^{-\frac{1}{2}}\ast(-\Delta_{\mathbb{H}}-9/4)^{-\frac{1}{2}}=(-\Delta_{\mathbb{H}}-9/4)^{-1}$,
we have, by  (\ref{3.1}),
\begin{equation*}
\begin{split}
(-\Delta_{\mathbb{H}}-9/4)^{-\frac{1}{2}}\ast(-\Delta_{\mathbb{H}}+\alpha)^{-\frac{1}{2}}\leq&(-\Delta_{\mathbb{H}}-9/4)^{-\frac{1}{2}}\ast(-\Delta_{\mathbb{H}}-9/4)^{-\frac{1}{2}}\\
=&(-\Delta_{\mathbb{H}}-9/4)^{-1}.
\end{split}
\end{equation*}
Therefore, by Lemma \ref{lm4.1}, we get (\ref{4.6}).

Now we  prove (\ref{4.7}). By O'Neil's lemma (see \cite{on}),
\begin{equation}\label{c4.1}
\begin{split}
&[(-\Delta_{\mathbb{H}}-9/4)^{-\frac{1}{2}}\ast(-\Delta_{\mathbb{H}}+\alpha)^{-\frac{1}{2}}]^{\ast}(t)\\
\leq&\frac{1}{t}\int^{t}_{0}
[(-\Delta_{\mathbb{H}}-9/4)^{-\frac{1}{2}}]^{\ast}(s)ds\cdot\int^{t}_{0}
[(-\Delta_{\mathbb{H}}+\alpha)^{-\frac{1}{2}}]^{\ast}(s)ds+\\
&\int^{\infty}_{t}[(-\Delta_{\mathbb{H}}-9/4)^{-\frac{1}{2}}]^{\ast}(s)
[(-\Delta_{\mathbb{H}}+\alpha)^{-\frac{1}{2}}]^{\ast}(s)ds.
\end{split}
\end{equation}
By Lemma \ref{lm4.2}, it is easy to check
\begin{equation}\label{b4.9}
\begin{split}
\int^{t}_{0}[(-\Delta_{\mathbb{H}}+\alpha)^{-\frac{1}{2}}]^{\ast}(s)ds\lesssim& \int^{2}_{0}\frac{2^{\frac{1}{4}}}{8\sqrt{\pi}s^{\frac{3}{4}}}(1+A_{2}\sqrt{s})ds+
\int^{t}_{2}s^{-1-\frac{1}{3}\epsilon_{0}}ds\lesssim 1,\;\;\;t>2.
\end{split}
\end{equation}
Also by Lemma \ref{lm4.2} and (\ref{4.5}), we have, for $t>2$,
\begin{equation}\label{b4.7}
\begin{split}
\int^{t}_{0}[(-\Delta_{\mathbb{H}}-9/4)^{-\frac{1}{2}}]^{\ast}(s)ds\lesssim&\int^{2}_{0}\frac{2^{\frac{1}{4}}}{8\sqrt{\pi}s^{\frac{3}{4}}}(1+A_{2}\sqrt{s})ds+
\int^{t}_{2}\frac{1}{\sqrt{s}\ln s}ds\\
\lesssim& 1+\int^{t}_{2}\frac{1}{\sqrt{s}\ln s}ds.
\end{split}
\end{equation}
Notice that, by L'Hospital's law,
\begin{equation*}
\begin{split}
\lim_{t\rightarrow\infty}\frac{\int^{t}_{2}\frac{1}{\sqrt{s}\ln s}ds}{\frac{\sqrt{t}}{\ln t}}=\lim_{t\rightarrow\infty}\frac{2}{1-\frac{1}{\ln t}}=2.
\end{split}
\end{equation*}
We have, $\int^{t}_{2}\frac{1}{\sqrt{s}\ln s}ds\lesssim \frac{\sqrt{t}}{
\ln t}$, $t>2$. Therefore,
by (\ref{b4.7}),
\begin{equation}\label{b4.8}
\begin{split}
\int^{t}_{0}[(-\Delta_{\mathbb{H}}-9/4)^{-\frac{1}{2}}]^{\ast}(s)ds\lesssim 1+\frac{\sqrt{t}}{
\ln t}
\lesssim&\frac{\sqrt{t}}{
\ln t},\;\;\;t>2.
\end{split}
\end{equation}
Similarly, by Lemma \ref{lm4.2},  for $t>2$,
\begin{equation}\label{b4.10}
\begin{split}
\int^{\infty}_{t}[(-\Delta_{\mathbb{H}}-9/4)^{-\frac{1}{2}}]^{\ast}(s)
[(-\Delta_{\mathbb{H}}+\alpha)^{-\frac{1}{2}}]^{\ast}(s)ds\lesssim&\int^{\infty}_{t}\frac{1}{s^{\frac{3}{2}+\frac{1}{3}\epsilon_{0}}\ln
s}ds
\lesssim\frac{1}{t^{\frac{1}{2}+\frac{1}{3}\epsilon_{0}}\ln
t}.
\end{split}
\end{equation}
Combing (\ref{c4.1}), (\ref{b4.9}), (\ref{b4.8}) and  (\ref{b4.10}) yields, for $t>2$,
\begin{equation*}
\begin{split}
[(-\Delta_{\mathbb{H}}-9/4)^{-\frac{1}{2}}\ast(-\Delta_{\mathbb{H}}+\alpha)^{-\frac{1}{2}}]^{\ast}(t)\lesssim&\frac{1}{t}\cdot\frac{\sqrt{t}}{
\ln t}\cdot 1+\frac{1}{t^{\frac{1}{2}+\frac{1}{3}\epsilon_{0}}\ln
t}\lesssim\frac{1}{\sqrt{t}
\ln t}.
\end{split}
\end{equation*}
The desired result follows.
\end{proof}

\section{Proofs of main theorems}

Firstly,  we shall prove Theorem 1.4. The main idea is to decompose  the whole space by the level set of the functions under consideration and derive the global inequality on the whole space from the local ones. This idea was initially  developed different settings by Lam and the first author to derive a global
Trudinger-Moser  inequality
from a local one (see \cite{ll,l}).\\

\textbf{Proof of Theorem \ref{th1.1}}   Let $u\in
C^{\infty}_{0}(\mathbb{B})$ be such that
\[
\int_{\mathbb{B}^{4}}\left(-\Delta_{\mathbb{H}}-9/4\right)(-\Delta_{\mathbb{H}}+\alpha)u\cdot
udV\leq1,
\]
We have, by (\ref{1.4}),
 \begin{equation}\label{5.1}
\begin{split}
&\int_{\mathbb{B}^{4}}\left(-\Delta_{\mathbb{H}}-9/4\right)(-\Delta_{\mathbb{H}}+\alpha)u\cdot
udV\\
\geq&\left(\frac{9}{4}+\alpha\right)\left(\int_{\mathbb{B}^{4}}|\nabla_{\mathbb{H}}u|^{2}dV-\frac{9}{4}\int_{\mathbb{B}^{4}}|u|^{2}dV\right)\\
\geq&\frac{9}{4}\left(\int_{\mathbb{B}^{4}}|\nabla_{\mathbb{H}}u|^{2}dV-\frac{9}{4}\int_{\mathbb{B}^{4}}|u|^{2}dV\right).
\end{split}
\end{equation}
Combing (\ref{5.1})  and the following Hardy-Sobolev inequality on
$\mathbb{B}^{4}$ (see e.g. \cite{m1})
 \begin{equation}\label{5.2}
\begin{split}
\int_{\mathbb{B}^{4}}|\nabla_{\mathbb{H}}u|^{2}dV-\frac{9}{4}\int_{\mathbb{B}^{4}}|u|^{2}dV\geq
C_{4}\left(\int_{\mathbb{B}^{4}}|u|^{4}dV\right)^{\frac{1}{2}},
\end{split}
\end{equation}
we have,
 \begin{equation}\label{5.3}
\begin{split}
1\geq\int_{\mathbb{B}^{4}}\left(-\Delta_{\mathbb{H}}-9/4\right)(-\Delta_{\mathbb{H}}+\alpha)u\cdot
udV\geq&
C^{-1}_{5}\left(\int_{\mathbb{B}^{4}}|u|^{4}dV\right)^{\frac{1}{2}},
\end{split}
\end{equation}
where $C_{5}=\frac{4}{9 C_{4}}$.

 Set $\Omega(u)=\{x\in\mathbb{B}^{4}:|u(x)|\geq1\}$. By (\ref{5.3}),
 we have
\begin{equation}\label{5.4}
\begin{split}
|\Omega(u)|=&\int_{\Omega(u)}dV\leq\int_{\mathbb{B}}|u(z)|^{4}dV\leq
C^{2}_{5}.
\end{split}
\end{equation}
We write
\begin{equation}\label{4.2}
\begin{split}
&\int_{\mathbb{B}^{4}}(e^{32\pi^{2} u^{2}}-1-32\pi^{2} u^{2})dV\\
=&\int_{\Omega(u)}(e^{32\pi^{2} u^{2}}-1-32\pi^{2} u^{2})dV+
\int_{\mathbb{B}^{4}\setminus\Omega(u)}(e^{32\pi^{2} u^{2}}-1-32\pi^{2} u^{2})dV\\
\leq&\int_{\Omega(u)}e^{32\pi^{2} u^{2}}dV+
\int_{\mathbb{B}^{4}\setminus\Omega(u)}(e^{32\pi^{2}
u^{2}}-1-32\pi^{2} u^{2})dV.
\end{split}
\end{equation}
Notice that on the domain $\mathbb{B}^{4}\setminus\Omega(u)$, we
have $|u(x)|<1$. Thus,
\begin{equation}\label{5.6}
\begin{split}
&\int_{\mathbb{B}^{4}\setminus\Omega(u)}(e^{32\pi^{2} u^{2}}-1-32\pi^{2} u^{2})dV\\
=&\int_{\mathbb{B}^{4}\setminus\Omega(u)}\sum^{\infty}_{n=2}\frac{(32\pi^{2} u^{2})^{n}}{n!}dV\\
\leq&\int_{\mathbb{B}\setminus\Omega(u)}\sum^{\infty}_{n=2}\frac{(32\pi^{2} )^{n}u^{4}}{n!}dV\\
\leq&\sum^{\infty}_{n=2}\frac{(32\pi^{2} )^{n}}{n!}\int_{\mathbb{B}^{4}}|u(x)|^{4}dV\\
\leq&e^{32\pi^{2}} C^{2}_{5}.
\end{split}
\end{equation}

To finish the proof, it is enough to show
$\int_{\Omega(u)}e^{32\pi^{2} u^{2}}dV$ is bounded by some universal
constant.
By (\ref{5.4}), we may assume
\begin{equation}\label{b5.2}
\begin{split}
|\Omega(u)|\leq \Omega_{0}
\end{split}
\end{equation}
for some constant $\Omega_{0}$ which is independent of $u$.
Now rewrite
\begin{equation}\label{5.7}
\begin{split}
\int_{\mathbb{B}^{4}}\left(-\Delta_{\mathbb{H}}-9/4\right)(-\Delta_{\mathbb{H}}+\alpha)u\cdot
udV=\int_{\mathbb{B}^{4}}\left|\sqrt{\left(-\Delta_{\mathbb{H}}-9/4\right)(-\Delta_{\mathbb{H}}+\alpha)}
\;u\right|^{2}dV
\end{split}
\end{equation}
and set
\[
v=\sqrt{\left(-\Delta_{\mathbb{H}}-9/4\right)(-\Delta_{\mathbb{H}}+\alpha)}\;
u=\sqrt{-\Delta_{\mathbb{H}}-9/4}(\sqrt{-\Delta_{\mathbb{H}}+\alpha}\;u).
\]
Then
\begin{equation}\label{b5.1}
\begin{split}
\int_{\mathbb{B}^{4}}|v|^{2}dV=\int_{\mathbb{B}^{4}}\left|\sqrt{-\Delta_{\mathbb{H}}-9/4}(\sqrt{-\Delta_{\mathbb{H}}+\alpha}\;u)\right|^{2}dV\leq1.
\end{split}
\end{equation}
Furthermore, by
(\ref{2.2}),  we can write $u$ as a potential
\[
u=(v\ast(-\Delta_{\mathbb{H}}+\alpha)^{-^{\frac{1}{2}}})\ast(-\Delta_{\mathbb{H}}-9/4)^{-\frac{1}{2}}=v\ast((-\Delta_{\mathbb{H}}+\alpha)^{-^{\frac{1}{2}}}\ast(-\Delta_{\mathbb{H}}-9/4)^{-\frac{1}{2}}).
\]
Let
$\varphi_{1}=(-\Delta_{\mathbb{H}}-9/4)^{-\frac{1}{2}}\ast(-\Delta_{\mathbb{H}}+\alpha)^{-\frac{1}{2}}$. Then $u=v\ast\varphi_{1}$.
By (\ref{b5.2}) and O'Neil's lemma,
\begin{equation*}
\begin{split}
&\int_{\Omega(u)}e^{32\pi^{2} u^{2}}dV=\int^{|\Omega(u)|}_{0}\exp(32\pi^{2}|u^{\ast}(t)|^{2})dt\leq\int^{\Omega_{0}}_{0}\exp(32\pi^{2}|u^{\ast}(t)|^{2})dt\\
\leq&\int^{\Omega_{0}}_{0}\exp\left(32\pi^{2}\left|\frac{1}{t}\int^{t}_{0}v^{\ast}(s)ds\int^{t}_{0}\varphi_{1}^{\ast}(s)ds+
\int^{\infty}_{t}v^{\ast}(s)\varphi_{1}^{\ast}(s)ds\right|^{2}\right)dt
\\
=&\Omega_{0}\int^{\infty}_{0}\exp\left(-t+32\pi^{2}\left|\frac{1}{\Omega_{0}e^{-t}}\int^{\Omega_{0}e^{-t}}_{0}v^{\ast}(s)ds\int^{\Omega_{0}e^{-t}}_{0}\varphi_{1}^{\ast}(s)ds+
\int^{\infty}_{\Omega_{0}e^{-t}}v^{\ast}(s)\varphi_{1}^{\ast}(s)ds\right|^{2}\right)dt.
\end{split}
\end{equation*}
To get the last equation, we use the substitution $t:=\Omega_{0}e^{-t}$.
Next,  we change the variables
\begin{equation*}
\begin{split}
\psi(t)=&\sqrt{\Omega_{0}e^{-t}}v^{\ast}(\Omega_{0}e^{-t});\\
\varphi(t)=&\sqrt{32\pi^{2}\Omega_{0}e^{-t}}\varphi_{1}^{\ast}(\Omega_{0}e^{-t}).
\end{split}
\end{equation*}
It is easy to check
\begin{equation*}
\begin{split}
\int^{\infty}_{t}e^{-s/2}\psi(s)ds\int^{\infty}_{t}e^{-s/2}\varphi(s)ds=&\frac{\sqrt{32\pi^{2}}}{\Omega_{0}}
\int^{\Omega_{0}e^{-t}}_{0}v^{\ast}(s)ds\int^{\Omega_{0}e^{-t}}_{0}\varphi_{1}^{\ast}(s)ds;\\
\int^{t}_{-\infty}\psi(s)\varphi(s)ds=&\sqrt{32\pi^{2}}\int^{\infty}_{\Omega_{0}e^{-t}}v^{\ast}(s)\phi^{\ast}(s)ds.
\end{split}
\end{equation*}
Therefore,
\begin{equation*}
\begin{split}
&\int_{\Omega(u)}e^{32\pi^{2} u^{2}}dV\leq\int^{\Omega_{0}}_{0}\exp(32\pi^{2}|u^{\ast}(t)|^{2})dt\\
=&\Omega_{0}\int^{\infty}_{0}\exp\left(-t+\left(e^{t}\int^{\infty}_{t}e^{-s/2}\psi(s)ds\int^{\infty}_{t}e^{-s/2}\varphi(s)ds+\int^{t}_{-\infty}\psi(s)\varphi(s)ds\right)^{2}\right)dt\\
=&\Omega_{0}\int^{\infty}_{0}e^{-F(t)}dt,
\end{split}
\end{equation*}
where
\[
F(t)=t-\left(e^{t}\int^{\infty}_{t}e^{-s/2}\psi(s)ds\int^{\infty}_{t}e^{-s/2}\varphi(s)ds+\int^{t}_{-\infty}\psi(s)\varphi(s)ds\right)^{2}.
\]
Set
\begin{equation}\label{b5.8}
\begin{split}
a(s,t)=\left\{
         \begin{array}{ll}
           \varphi(s), & \hbox{$s<t$;} \\
           e^{t}(\int^{\infty}_{t}e^{-r/2}\varphi(r)dr)e^{-s/2}, & \hbox{$s>t$.}
         \end{array}
       \right.
\end{split}
\end{equation}
Then
\[
F(t)=t-\left(\int^{\infty}_{-\infty}a(s,t)\psi(s)ds\right)^{2}.
\]
To complete the proof, we need to show that there exists a constant $C$ which is independent of  $\psi$ such that
\[
\int^{\infty}_{0}e^{-F(t)}dt<C.
\]
This will be done in the following Lemma \ref{lm5.1}. The proof of Theorem \ref{th1.1} is thereby completed.

\begin{lemma}\label{lm5.1}
Let $\psi(s)$, $a(s,t)$ and $F(t)$ be as above. Then there a constant $C_{6}$ which is independent of  $\psi$ such that
$\int^{\infty}_{0}e^{-F(t)}dt<C_{6}.$
\end{lemma}
\begin{proof}
The proof is similar to that in Adams' paper \cite{ad}. Notice that
\[
\int^{\infty}_{0}e^{-F(t)}dt=\int^{\infty}_{-\infty}|E_{\lambda}|e^{-\lambda}d\lambda,
\]
where $E_{\lambda}=\{t\geq0: F(t)\leq\lambda\}$ and $|E_{\lambda}|$ is the Lebesgue measure of $E_{\lambda}$. It is enough to show the following two facts:

(a) There exists a constant $c\geq0$ which is independent of $\psi$ such that $\inf\limits_{t\geq0} F(t)\geq -c;$

(b) There exist constants $B_{1}$ and $B_{2}$ which are both independent of $\psi$ such that $|E_{\lambda}|\leq B_{1}|\lambda|+B_{2}$.

Firstly, we prove (a). Without loss of generality, we assume $\Omega_{0}>2$ so that $\ln\frac{\Omega_{0}}{2}>0$.  By Lemma \ref{lm4.3},
\begin{equation}\label{c5.1}
\begin{split}
\varphi(t)=&\sqrt{32\pi^{2}\Omega_{0}e^{-t}}\varphi_{1}^{\ast}(\Omega_{0}e^{-t})\leq1+A_{1}\sqrt[4]{\Omega_{0}}e^{-\frac{t}{4}},\;\;\;\; t\in\mathbb{R};\\
\varphi(t)=&\sqrt{32\pi^{2}\Omega_{0}e^{-t}}\varphi_{1}^{\ast}(\Omega_{0}e^{-t})\lesssim\frac{1}{\ln\Omega_{0}-t}\leq \frac{C}{1-t},\;\;\;\; t\leq\ln\frac{\Omega_{0}}{2},
\end{split}
\end{equation}
where $C>0$ is a constant which is independent of $\psi$. Therefore, by (\ref{b5.8}),  for $t\geq0$,
\begin{equation*}
\begin{split}
\int^{\infty}_{-\infty}a(s,t)^{2}ds=&\int^{0}_{-\infty}a(s,t)^{2}ds+\int^{t}_{0}a(s,t)^{2}ds+\int^{\infty}_{t}a(s,t)^{2}ds\\
=&\int^{0}_{-\infty}\varphi^{2}(s)ds+\int^{t}_{0}\varphi^{2}(s)ds+ e^{2t}\left(\int^{\infty}_{t}e^{-r/2}\varphi(r)dr\right)^{2}\int^{\infty}_{t}e^{-s}ds\\
\leq&C\int^{0}_{-\infty}\frac{1}{(1-s)^{2}}ds+\int^{t}_{0}(1+A_{1}\sqrt[4]{\Omega_{0}}e^{-\frac{s}{4}})^{2}ds+\\
&e^{t}\left(\int^{\infty}_{t}e^{-r/2}(1+A_{1}\sqrt[4]{\Omega_{0}}e^{-\frac{r}{4}})dr\right)^{2}\\
\leq&t+c,
\end{split}
\end{equation*}
where $c>0$ is independent of $\psi$. Thus, by Cauchy-Schwarz inequality,
\begin{equation*}
\begin{split}
F(t)=&t-\left(\int^{\infty}_{-\infty}a(s,t)\psi(s)ds\right)^{2}\\
\geq& t-\int^{\infty}_{-\infty}a(s,t)^{2}ds\int^{\infty}_{-\infty}\psi^{2}(s)ds\\
=&t-\int^{\infty}_{-\infty}a(s,t)^{2}ds\int^{\infty}_{-\infty}\Omega_{0}e^{-t}|v^{\ast}(\Omega_{0}e^{-t})|^{2}ds\\
\geq&t-(t+c)\int_{\mathbb{B}^{4}}|v|^{2}dV\\
=& t-(t+c)=-c.
\end{split}
\end{equation*}

Secondly, we prove (b). Let $R>0$ and suppose $E_{\lambda}\cap[R,\infty)\neq{\o}$. Take $t_{1}, t_{2}\in E_{\lambda}\cap[R,\infty)\neq{\o}$, $t_{1}<t_{2}$. Then
\begin{equation}\label{c5.3}
\begin{split}
t_{2}-\lambda\leq&\left(\int^{\infty}_{-\infty}a(s,t_{2})\psi(s)dt\right)^{2}\\
=&\left(\int^{t_{1}}_{-\infty}a(s,t_{2})\psi(s)ds+\int^{t_{2}}_{t_{1}}a(s,t_{2})\psi(s)ds
+\int^{\infty}_{t_{2}}a(s,t_{2})\psi(s)ds\right)^{2}\\
=&\left(\int^{t_{1}}_{-\infty}\varphi(s)\psi(s)ds+\int^{t_{2}}_{t_{1}}\varphi(s)\psi(s)ds
+\int^{\infty}_{t_{2}}a(s,t_{2})\psi(s)ds\right)^{2}.
\end{split}
\end{equation}

Set $L(t)=(\int^{\infty}_{t}\psi^{2}(s)ds)^{\frac{1}{2}}$. Then
\[L(t)\leq\left(\int^{\infty}_{\infty}\psi^{2}(s)ds\right)^{\frac{1}{2}}=\left(\int^{\infty}_{-\infty}\Omega_{0}e^{-t}|v^{\ast}(\Omega_{0}e^{-t})|^{2}ds\right)^{\frac{1}{2}}
=\left(\int_{\mathbb{B}^{4}}|v|^{2}dV\right)^{\frac{1}{2}}\leq1\]
Therefore, by (\ref{c5.1}) and Cauchy-Schwarz inequality,
\begin{equation}\label{c5.4}
\begin{split}
\left(\int^{t_{1}}_{-\infty}\varphi(s)\psi(s)ds\right)^{2}\leq&\int^{t_{1}}_{-\infty}\varphi^{2}(s)ds\cdot \int^{\infty}_{\infty}\psi^{2}(s)ds\\
\leq&\int^{t_{1}}_{-\infty}\varphi^{2}(s)ds=\int^{0}_{-\infty}\varphi^{2}(s)ds+\int^{t_{1}}_{0}\varphi^{2}(s)ds\\
\leq& C\int^{0}_{-\infty}\frac{1}{(1-s)^{2}}ds+\int^{t_{1}}_{0}(1+A_{1}\sqrt[4]{\Omega_{0}}e^{-\frac{s}{4}})^{2}ds\\
\leq&t_{1}+b_{1};
\end{split}
\end{equation}
\begin{equation}\label{1c5.4}
\begin{split}
\left(\int^{t_{2}}_{t_{1}}\varphi(s)\psi(s)ds\right)^{2}\leq&\int^{t_{2}}_{t_{1}}\varphi^{2}(s)ds\cdot L^{2}(t_{1})\\
\leq&\int^{t_{2}}_{t_{1}}(1+A_{1}\sqrt[4]{\Omega_{0}}e^{-\frac{s}{4}})^{2}ds\cdot L^{2}(t_{1})\\
\leq& (t_{2}-t_{1}+b_{2})L^{2}(t_{1});
\end{split}
\end{equation}
\begin{equation}\label{2c5.4}
\begin{split}
\left(\int^{\infty}_{t_{2}}a(s,t_{2})\psi(s)ds\right)^{2}\leq&\int^{\infty}_{t_{2}}a(s,t_{2})^{2}ds \cdot L^{2}(t_{2})\leq\int^{\infty}_{t_{2}}a(s,t_{2})^{2}ds\cdot L^{2}(t_{1})\\
=&e^{2t_{2}}\left(\int^{\infty}_{t_{2}}e^{-r/2}\varphi(r)dr\right)^{2}\int^{\infty}_{t_{2}}e^{-s}ds\cdot L^{2}(t_{1})\\
\leq&e^{t_{2}}\left(\int^{\infty}_{t_{2}}e^{-r/2}((1+A_{1}\sqrt[4]{\Omega_{0}}e^{-\frac{r}{4}})dr\right)^{2}\cdot L^{2}(t_{1})\\
\leq&b_{3}L^{2}(t_{1}),
\end{split}
\end{equation}
where $b_{1},b_{2},b_{3}$ are constants independent of $t_{1}$ and $t_{2}$.
Combing (\ref{c5.3}) and (\ref{c5.4})-(\ref{2c5.4}) yields
\begin{equation*}
\begin{split}
t_{2}-\lambda\leq&\left\{(t_{1}+b_{1})^{\frac{1}{2}}+[(t_{2}-t_{1}+b_{2})^{\frac{1}{2}}+b_{3}]L(t_{1})\right\}^{2}\\
\leq&\left\{(t_{1}+b_{1})^{\frac{1}{2}}+[(t_{2}-t_{1})^{\frac{1}{2}}+b_{4}]L(t_{1})\right\}^{2},
\end{split}
\end{equation*}
where $b_{4}$ is a constant independent of $t_{1}$ and $t_{2}$.
The rest of the proof is the same as that in \cite{ad} and thus the proof is completed.
\end{proof}

\textbf{Proof of Theorem \ref{th1.6}} Let $u\in
C^{\infty}_{0}(\mathbb{B}^{4})$ be such that
\[
\int_{\mathbb{B}^{4}}|\Delta
u|^{2}dx-\lambda\int_{\mathbb{B}^{4}}\frac{u^{2}}{(1-|x|^{2})^{4}}dx\leq1.
\]
Then
\begin{equation*}
\begin{split}
(9-\lambda)\int_{\mathbb{B}^{4}}\frac{u^{2}}{(1-|x|^{2})^{4}}dx\leq&\int_{\mathbb{B}^{4}}|\Delta
u|^{2}dx-\lambda\int_{\mathbb{B}^{4}}\frac{u^{2}}{(1-|x|^{2})^{4}}dx\leq1.
\end{split}
\end{equation*}
Therefore, by Corollary \ref{c2},
\begin{equation*}
\begin{split}
\int_{\mathbb{B}^{4}}(e^{32\pi^{2}
u^{2}}-1)dV
=&16\int_{\mathbb{B}^{4}}\frac{e^{32\pi^{2}
u^{2}}-1-32\pi^{2}u^{2}}{(1-|x|^{2})^{4}}dx+16\cdot32\pi^{2}\int_{\mathbb{B}^{4}}\frac{u^{2}}{(1-|x|^{2})^{4}}dx\\
\leq& C_{1}+\frac{16\cdot32\pi^{2}}{9-\lambda}.
\end{split}
\end{equation*}
The desired results follows.
\\

Before the proof of Theorem 1.9, we need the following improved
Hardy inequality.
\begin{lemma}
There exists a constant $C_{7}>0$ such that for all $u\in
C^{\infty}_{0}(\mathbb{B}^{4})$,
\[
\int_{\mathbb{B}^{4}}|\Delta
u|^{2}dx-9\int_{\mathbb{B}^{4}}\frac{u^{2}}{(1-|x|^{2})^{4}}dx\geq
C_{7} \int_{\mathbb{B}^{4}}u^{2}dx.
\]
\end{lemma}
\begin{proof} It is enough to show
\[\int_{\mathbb{B}^{4}}|\nabla_{\mathbb{H}}u|^{2}dV-\frac{9}{4}\int_{\mathbb{B}^{4}}u^{2}dV\geq
C_{8} \int_{\mathbb{B}^{4}}u^{2}dx,\] since, by Plancherel formula,
\begin{equation*}
\begin{split}
&\int_{\mathbb{B}^{4}}|\Delta
u|^{2}dx-9\int_{\mathbb{B}^{4}}\frac{u^{2}}{(1-|x|^{2})^{4}}dx\\
=&\int_{\mathbb{B}^{4}}\left(-\Delta_{\mathbb{H}}-9/4\right)(-\Delta_{\mathbb{H}}+1/2)u\cdot
udV\\
=&D_{4}\int^{+\infty}_{-\infty}\int_{\mathbb{S}^{n-1}}\frac{\lambda^{2}}{4}\cdot\left(\frac{9+\lambda^{2}}{4}+\frac{1}{2}\right)|\widehat{f}(\lambda,\zeta)|^{2}|\mathfrak{c}(\lambda)|^{-2}d\lambda d\sigma(\varsigma)\\
\geq&\left(\frac{9}{4}+\frac{1}{2}\right)D_{4}\int^{+\infty}_{-\infty}\int_{\mathbb{S}^{n-1}}\frac{\lambda^{2}}{4}|\widehat{f}(\lambda,\zeta)|^{2}|\mathfrak{c}(\lambda)|^{-2}d\lambda d\sigma(\varsigma)\\
=&\left(\frac{9}{4}+\frac{1}{2}\right)\int_{\mathbb{B}^{4}}\left(-\Delta_{\mathbb{H}}-9/4\right)u\cdot
udV\\
=&\left(\frac{9}{4}+\frac{1}{2}\right)\left(\int_{\mathbb{B}^{4}}|\nabla_{\mathbb{H}}u|^{2}dV-\frac{9}{4}\int_{\mathbb{B}^{4}}|u|^{2}dV\right).
\end{split}
\end{equation*}

Set $u=(1-|x|^{2})f$. Then $f\in C^{\infty}_{0}(\mathbb{B}^{4})$ and
\[
|\nabla u|^{2}=(1-|x|^{2})^{2}|\nabla
f|^{2}+\frac{1}{2}\langle\nabla f^{2},\nabla
(1-|x|^{2})^{2}\rangle+f^{2} \cdot4|x|^{2}.
\]
Therefore,
\begin{equation*}
\begin{split}
\int_{\mathbb{B}^{4}}\frac{|\nabla
u|^{2}}{(1-|x|^{2})^{2}}dx=&\int_{\mathbb{B}^{4}}|\nabla
f|^{2}dx+\int_{\mathbb{B}^{4}}f^{2}\frac{4|x|^{2}}{(1-|x|^{2})^{2}}dx-\frac{1}{2}\int_{\mathbb{B}^{4}}f^{2}\Delta\ln(1-|x|^{2})^{2}dx\\
=&\int_{\mathbb{B}^{4}}|\nabla
f|^{2}dx+8\int_{\mathbb{B}^{4}}\frac{f^{2}}{(1-|x|^{2})^{2}}dx.
\end{split}
\end{equation*}
Recall that the improved Hardy inequality (see e.g.\cite{bm,wy})
\[
\int_{\mathbb{B}^{4}}|\nabla
f|^{2}dx-\int_{\mathbb{B}^{4}}\frac{f^{2}}{(1-|x|^{2})^{2}}dx\geq
C_{9}\int_{\mathbb{B}^{4}}f^{2}dx.
\]
We have
\begin{equation*}
\begin{split}
\int_{\mathbb{B}^{4}}|\nabla_{\mathbb{H}}u|^{2}dV-\frac{9}{4}\int_{\mathbb{B}^{4}}u^{2}dV=&4\left(\int_{\mathbb{B}^{4}}\frac{|\nabla
u|^{2}}{(1-|x|^{2})^{2}}dx-9\int_{\mathbb{B}^{4}}\frac{u^{2}}{(1-|x|^{2})^{2}}dx\right)\\
=&4\left(\int_{\mathbb{B}^{4}}|\nabla
f|^{2}dx-\int_{\mathbb{B}^{4}}\frac{f^{2}}{(1-|x|^{2})^{2}}dx\right)\\
\geq&
4C_{9}\int_{\mathbb{B}^{4}}f^{2}dx=4C_{9}\int_{\mathbb{B}^{4}}\frac{u^{2}}{1-|x|^{2}}dx\\
\geq&4C_{9}\int_{\mathbb{B}^{4}}u^{2}dx.
\end{split}
\end{equation*}
This completes the proof.

\end{proof}

\textbf{Proof of Theorem \ref{th1.7}} Let $u\in
C^{\infty}_{0}(\mathbb{B}^{4})$ be such that
\[
\int_{\mathbb{B}^{4}}|\Delta
u|^{2}dx-9\int_{\mathbb{B}^{4}}\frac{u^{2}}{(1-|x|^{2})^{4}}dx\leq1,
\]
 By Corollary 1.5, there exist a positive constant $C_{1}>0$ which is independent of $u$ such that
\[
\int_{\mathbb{B}^{4}}\frac{(e^{32\pi u^{2}}-1-32\pi^{2}
u^{2})}{(1-|x|^{2})^{4}}dx\leq C_{1}.
\]
Therefore, by Lemma 5.2,
\begin{equation*}
\begin{split}
\int_{\mathbb{B}^{4}}e^{32\pi^{2} u^{2}}dx=&\int_{\mathbb{B}^{4}}(e^{32\pi^{2} u^{2}}-1-32\pi^{2} u^{2})dx++\int_{\mathbb{B}^{4}}dx+32\pi^{2}\int_{\mathbb{B}^{4}}u^{2}dx\\
\leq&\int_{\mathbb{B}^{4}}\frac{(e^{32\pi^{2} u^{2}}-1-32\pi^{2} u^{2})}{(1-|x|^{2})^{4}}dx+\int_{\mathbb{B}^{4}}dx+32\pi^{2}\int_{\mathbb{B}^{4}}u^{2}dx\\
\leq& C_{1}+\int_{\mathbb{B}^{4}}dx+32\pi^{2}C^{-1}_{7}.
\end{split}
\end{equation*}
The desired result follows.

\end{document}